\def\benm{\begin{enumerate}}
\def\eenm{\end{enumerate}}

\def\bal{\begin{align}}
\def\eal{\end{align}}

\def\SO{{\rm SO}}

\documentclass{amsart}
\usepackage{amsmath, amsthm, amscd, amsfonts,amssymb, graphicx}

\newtheorem{theorem}{Theorem}[section]

\theoremstyle{definition}

\newtheorem{example}[theorem]{Example}

\newtheorem{proposition}[theorem]{Proposition}
\newtheorem{corollary}[theorem]{Corollary}

\theoremstyle{remark}
\newtheorem{remark}[theorem]{Remark}

\numberwithin{equation}{section}

\begin{document}

\title[Continuous Gabor transform for semi-direct product of locally compact groups]{Continuous Gabor transform for semi-direct product of locally compact groups}

\author[Arash Ghaani Farashahi]{Arash Ghaani Farashahi}
\address{Department of Pure Mathematics, Faculty of Mathematical sciences, Ferdowsi University of
Mashhad (FUM), P. O. Box 1159, Mashhad 91775, Iran.}
\address{Center of Excellence in Analysis on Algebraic Structures (CEAAS),
Ferdowsi University of Mashhad (FUM), P. O. Box 1159, Mashhad 91775, Iran.}
\email{ghaanifarashahi@hotmail.com}

\curraddr{}




\date{}


\thanks{E-mail addresses: ghaanifarashahi@hotmail.com (Arash Ghaani Farashahi)}

\begin{abstract}
Let $H$ be a locally compact group, $K$ be an LCA group, $\tau:H\to Aut(K)$ be a continuous homomorphism and $G_\tau=H\ltimes_\tau K$ be the semi-direct product of $H$ and $K$ with respect to the continuous homomorphism $\tau$. In this article we introduce the $\tau\times\widehat{\tau}$-time frequency group $G_{\tau\times\widehat{\tau}}$. We define the $\tau\times\widehat{\tau}$-continuous Gabor transform of $f\in L^2(G_\tau)$ with respect to a window function $u\in L^2(K)$ as a function defined on $G_{\tau\times\widehat{\tau}}$.
It is also shown that the $\tau\times\widehat{\tau}$-continuous Gabor transform satisfies the Plancherel Theorem and reconstruction formula. This approach is tailored for choosing elements of $L^2(G_\tau)$ as a window function.
Finally, we illustrate application of these methods in the case of some well-known semi-direct product groups.
\end{abstract}

\maketitle

\section{\bf{Introduction}}

In \cite{Gab} Gabor used translations and modulations of the Gaussian signal to represent one dimensional signals. The Gabor transform, named after Gabor, is a special case of the short-time Fourier transform (STFT). It is used to determine the sinusoidal frequency and phase content of local sections of a signal as it changes over time. The function to be transformed is first multiplied by a Gaussian function, which can be regarded as a window, and the resulting function is then transformed with a Fourier transform to derive the time-frequency analysis. The window function term means that the signal near the time being analyzed will have higher weight. The Gabor transform of a signal $x(t)$ is precisely defined by;
\begin{equation}\label{SG}
G\{x\}(y,\omega)=\int_{-\infty}^{+\infty}x(t)e^{-\pi(t-y)^2}e^{-2\pi i\omega t}dt.
\end{equation}
Due to (\ref{SG}) the Gabor transform of a signal $x(t)$ is a function defined on $\mathbb{R}\times\widehat{\mathbb{R}}$
called the time-frequency plane. There is also standard extension of the continuous Gabor transform of a signal $x({\bf t})$ on $\mathbb{R}^n$ which is defined for $({\bf y},{\bf w})\in \mathbb{R}^n\times\widehat{\mathbb{R}^n}$ by (see \cite{FithZim, Gro1})
\begin{equation}\label{SGN}
G\{x\}({\bf y},{\bf w})=\int_{\mathbb{R}^n}x({\bf t})e^{-\pi\|{\bf t}-{\bf y}\|^2}e^{-2\pi i{\bf w}.{\bf t}}d{\bf t}.
\end{equation}
Since the theory of Gabor analysis based on the structure of translations and modulations (time-frequency plane), it is also possible to extend concepts of the Gabor theory to other locally compact abelian (LCA) groups.
For more explanation, we refer the reader to the monograph of ${\rm Gr\ddot{o}chenig}$ \cite{Gro2} or complete work of Feichtinger and ${\rm Str\ddot{o}hmer}$ \cite{FithStro}. The continuous Gabor transform for LCA groups is closely related to the Feichtinger-${\rm Gr\ddot{o}chenig}$ theory which is called the voice transform. In view of voice transform, continuous Gabor transform for an LCA group $G$ is precisely the voice transform generated by ${\rm Schr\ddot{o}dinger}$ representation of the Weyl-Heisenberg group associate with $G$ (see \cite{HeilW}).

Many locally compact groups which are used in mathematical physics and also various topics of engineering such as the Heisenberg group, affine group or Euclidean groups are non-abelian groups. Although most of them can be considered as a semi-direct product of an LCA group with another locally compact group. We recall that passing through the harmonic analysis of LCA groups (see \cite{FollH, 50}) to the classical harmonic analysis of non-abelian locally compact groups (see \cite{Dix, FollH, Lip, Seg1, Ta}), many useful results and basic concepts in abelian harmonic analysis collapse, which play important roles in the usual Gabor theory. If $G$ is a non-abelain locally compact group via a natural approach, modulation by a character will be replaced by a modulation by an equivalence class of an irreducible representation of $G$ and also the natural candidate for the generalization of the time frequency plane will be $G\times \widehat{G}$, where $\widehat{G}$ stands for the set of all equivalence class of irreducible continuous unitary representations of $G$. It is clear that this extension will not be appropriate from the the numerical computational and also application viewpoints. Thus, we need a new approach to find an appropriate generalization of the continuous Gabor transform which be useful and also efficient in application.

This article contains 5 sections. Section 2 is devoted to fix notations including a brief summary about harmonic analysis of semi-direct product of locally compact groups also standard Fourier analysis and Gabor analysis on LCA groups. In section 3 we assume that $H$ is a locally compact group and $K$ is an LCA group, $\tau:H\to Aut(K)$ is a continuous homomorphism and $G_\tau=H\ltimes_\tau{K}$. We define the  $\tau\times\widehat{\tau}$-time frequency group $G_{\tau\times\widehat{\tau}}$ and the $\tau\times\widehat{\tau}$-continuous Gabor transform of $f\in L^2(G_\tau)$ with respect to a window function $u\in L^2(K)$. We also prove a Plancherel and inversion formula for the $\tau\times\widehat{\tau}$-continuous Gabor transform. To choose elements of $L^2(G_\tau)$ as window functions we define the the $\tau\otimes\widehat{\tau}$-time frequency group $G_{\tau\otimes\widehat{\tau}}$ and also the $\tau\otimes\widehat{\tau}$-continuous Gabor transform in section 4. As an application, we study this theory on the affine group, Weyl-Heisenberg group and the Euclidean groups in section 5.

\section{\bf{Preliminaries and notations}}

Let $H$ and $K$ be locally compact groups with identity elements $e_H$ and $e_K$ respectively and left Haar measures $dh$ and $dk$ respectively,
also let $\tau:H\to Aut(K)$ be a homomorphism such that the map $(h,k)\mapsto \tau_h(k)$ is continuous from $H\times K$ onto $K$.
There is a natural topology, sometimes called
Braconnier topology, turning $Aut(K)$ into a Hausdorff topological group(not necessarily locally compact),
which is defined by the sub-base of identity neighbourhoods
\begin{equation}
\mathcal{B}(F,U)= \{ \alpha\in Aut(K): \alpha(k),\alpha^{-1}(k)\in Uk\ \forall k\in F\},
\end{equation}
where $F\subseteq K$ is compact and $U\subseteq K$ is an identity neighbourhood. Continuity of a homomorphism $\tau:H\to Aut(K)$ is equivalent with the continuity of the map $(h,k)\mapsto \tau_h(k)$ from $H\times K$ onto $K$ (see \cite{HO}).

The semi-direct product $G_\tau=H\ltimes_\tau K$ is the locally compact topological group with the underlying set $H\times K$ which is equipped by the product topology and also the group operation is defined by
\begin{equation}\label{0.1}
(h,k)\ltimes_\tau(h',k')=(hh',k\tau_h(k'))\hspace{0.5cm}{\rm and}\hspace{0.5cm}(h,k)^{-1}=(h^{-1},\tau_{h^{-1}}(k^{-1})).
\end{equation}
If $H_1=\{(h,e_K):h\in H\}$ and $K_1=\{(e_H,k):k\in K\}$ then $K_1$ is a closed normal subgroup and $H_1$ is a closed subgroup of $G_\tau$.
The left Haar measure of $G_\tau$ is $d\mu_{G_\tau}(h,k)=\delta(h)dhdk$ and the modular function $\Delta_{G_\tau}$ is $\Delta_{G_\tau}(h,k)=\delta(h)\Delta_H(h)\Delta_K(k)$, where the positive continuous homomorphism $\delta:H\to(0,\infty)$ is given by (15.29 of \cite{HR1})
\begin{equation}\label{t}
dk=\delta(h)d(\tau_h(k)).
\end{equation}
From now on, for all $p\ge 1$ we denote by $L^p(G_\tau)$ the Banach space $L^p(G_\tau,\mu_{G_\tau})$ and also $L^p(K)$ stands for $L^p(K,dk)$.
When $f\in L^p(G_\tau)$, for a.e. $h\in H$ the function $f_h$ defined on $K$ via $f_h(k):=f(h,k)$ belongs to $L^p(K)$ (see \cite{FollR}).

If $K$ is an LCA group all irreducible representations of $K$ are one-dimensional. Thus, if $\pi$ is an irreducible unitary representation of $K$ we have $\mathcal{H}_\pi=\mathbb{C}$ and also according to the Shur's Lemma there exists a continuous homomorphism $\omega$ of $K$ into the circle group $\mathbb{T}$ such that for each $k\in K$ and $z\in\mathbb{C}$ we have $\pi(k)(z)=\omega(k)z$. Such homomorphisms are called characters of $K$ and the set of all characters of $K$ denoted by $\widehat{K}$. If $\widehat{K}$ equipped by the topology of compact convergence on $K$ which coincides with the $w^*$-topology that $\widehat{K}$ inherits as a subset of $L^\infty(K)$, then $\widehat{K}$ with respect to the dot product of characters is an LCA group which is called the dual group of $K$.
The linear map $\mathcal{F}_K:L^1(K)\to \mathcal{C}(\widehat{K})$ defined by $v\mapsto \mathcal{F}_K(v)$ via
\begin{equation}\label{SF}
\mathcal{F}_K(v)(\omega)=\widehat{v}(\omega)=\int_Kv(k)\overline{\omega(k)}dk,
\end{equation}
is called the Fourier transform on $K$. It is a norm-decreasing $*$-homomorphism from $L^1(K)$ into $\mathcal{C}_0(\widehat{K})$ with a uniformly dense range in $\mathcal{C}_0(\widehat{K})$ (Proposition 4.13 of \cite{FollH}). If $\phi\in L^1(\widehat{K})$, the function defined $a.e.$ on $K$ by
\begin{equation}\label{parseval0}
\breve{\phi}(x)=\int_{\widehat{K}}\phi(\omega)\omega(x)d\omega,
\end{equation}
belongs to $L^\infty(K)$ and also for all $f\in L^1(K)$ we have the following orthogonality relation (Parseval formula);
\begin{equation}\label{parseval1}
\int_Kf(k)\overline{\breve{\phi}(k)}dk=\int_{\widehat{K}}\widehat{f}(\omega)\overline{\phi(\omega)}d\omega.
\end{equation}
The Fourier transform (\ref{SF}) on $L^1(K)\cap L^2(K)$ is an isometric transform and it extends uniquely to a unitary isomorphism from $L^2(K)$ onto $L^2(\widehat{K})$ (Theorem 4.25 of \cite{FollH}) also each $v\in L^1(K)$ with $\widehat{v}\in L^1(\widehat{K})$ satisfies the following Fourier inversion formula (Theorem 4.32 of \cite{FollH});
\begin{equation}
v(k)=\int_{\widehat{K}}\widehat{v}(\omega)\omega(k)d\omega\ {\rm for} \ {\rm a.e.}\ k\in K.
\end{equation}
The fundamental operator in standard Gabor theory is the time-frequency shift operator. If $K$ is an LCA group, the translation (time-shifts) operator is given by $T_sv(k)=v(k-s)$ for all $k,s\in K$ and also the modulation (frequency-shifts) operator is given by $M_\omega v(k)=\omega(k)v(k)$ for all $\omega\in\widehat{K}$, $k\in K$. The time-frequency shift operator is defined on the time-frequency plane (time-frequency group) $K\times\widehat{K}$ by $\varrho(k,\omega)=M_{\omega}T_k$ for all $(k,\omega)\in K\times\widehat{K}$.

Given an appropriate window function $u\in L^2(K)$ on $K$, the short time Fourier transform (STFT) or the continuous Gabor transform of $v\in L^2(K)$ is given by
\begin{equation}\label{CGT0}
V_uv(s,\omega)=\int_{K}v(k)\overline{u}(k-s)\overline{\omega(k)}dk=\langle v,\varrho(s,\omega)u\rangle_{L^2(K)}.
\end{equation}
The continuous Gabor transform (\ref{CGT0}) satisfies the following Plancherel formula
\begin{equation}\label{PL0}
\int_{K\times\widehat{K}}|V_uv(s,\omega)|^2dsd\omega=\|u\|_{L^2(K)}^2\|v\|_{L^2(K)}^2,
\end{equation}
for all $u,v\in L^2(K)$ (see \cite{Gro2}). If $u,u'\in L^2(K)$ with $\langle u,u'\rangle_{L^2(K)}\not=0$, then each $v\in L^2(K)$ satisfies the following inversion formula in the weak sense (see \cite{Gro1})
\begin{equation}
v=\langle u,u'\rangle_{L^2(K)}^{-1}\int_{K\times\widehat{K}}V_uv(k,\omega)\varrho(k,\omega)u'dkd\omega.
\end{equation}
If a window function $u\in L^2(K)$ has Fourier transform $\widehat{u}$ in $L^1(\widehat{K})$, then each $v\in L^2(K)$ with $\widehat{v}\in L^1(\widehat{K})$ satisfies the following inversion formula;
\begin{equation}\label{PI}
v(s)=\|u\|_{L^2(K)}^{-2}\int_{K\times\widehat{K}}V_uv(k,\omega)[\varrho(k,\omega)u](s)dkd\omega,
\end{equation}
for all $s\in K$.
\section{{\bf $\tau\times\widehat{\tau}$-continuous Gabor transform}}
Throughout this paper, let $H$ be a locally compact group and $K$ be an LCA group also let $\tau:H\to Aut(K)$ be a continuous homomorphism and $G_\tau=H\ltimes_\tau K$. For simplicity in notations we use $k^h$ instead of $\tau_h(k)$ for all $h\in H$ and $k\in K$.
In this section we introduce the $\tau\times\widehat{\tau}$-time frequency group and also we define the $\tau\times\widehat{\tau}$-continuous Gabor transform of $f\in L^2(G_\tau)$ with respect to a window function in $L^2(K)$.

Define $\widehat{\tau}:H\to Aut(\widehat{K})$ via $h\mapsto \widehat{\tau}_h$, given by
\begin{equation}\label{AAA}
\widehat{\tau}_h(\omega):=\omega_h=\omega\circ\tau_{h^{-1}}
\end{equation}
for all $\omega \in \widehat{K}$, where $\omega_h(k)=\omega(\tau_{h^{-1}}(k))$ for all $k\in K$.  If $\omega\in\widehat{K}$ and $h\in H$ we have $\omega_h\in\widehat{K}$, because for all $k,s\in K$ we have
\begin{align*}
\omega_h(ks)&=\omega\circ\tau_{h^{-1}}(ks)
\\&=\omega(\tau_{h^{-1}}(ks))
\\&=\omega(\tau_{h^{-1}}(k)\tau_{h^{-1}}(s))
=\omega(\tau_{h^{-1}}(k))\omega(\tau_{h^{-1}}(s))=\omega_h(k)\omega_h(s).
\end{align*}
According to (\ref{AAA}) for all $h\in H$ we have $\widehat{\tau}_h\in Aut(\widehat{K})$. Because, if $k\in K$ and $h\in H$ then for all $\omega,\eta\in\widehat{K}$ we have
\begin{align*}
\widehat{\tau}_h(\omega.\eta)(k)&=(\omega.\eta)_h(k)
\\&=(\omega.\eta)\circ\tau_{h^{-1}}(k)
\\&=\omega.\eta(\tau_{h^{-1}}(k))
\\&=\omega(\tau_{h^{-1}}(k))\eta(\tau_{h^{-1}}(k))
=\omega_h(k)\eta_h(k)=\widehat{\tau}_h(\omega)(k)\widehat{\tau}_h(\eta)(k).
\end{align*}
Also $h\mapsto \widehat{\tau}_h$ is a homomorphism from $H$ into $Aut(\widehat{K})$, cause if $h,t\in H$ then for all $\omega\in\widehat{K}$ and $k\in K$ we get
\begin{align*}
\widehat{\tau}_{th}(\omega)(k)&=\omega_{th}(k)
\\&=\omega(\tau_{(th)^{-1}}(k))
\\&=\omega(\tau_{h^{-1}}\tau_{t^{-1}}(k))
\\&=\omega_h(\tau_{t^{-1}}(k))
=\widehat{\tau}_h(\omega)(\tau_{t^{-1}}(k))=\widehat{\tau}_t[\widehat{\tau}_h(\omega)](k).
\end{align*}


Thus, we can prove the following theorem.


\begin{theorem}\label{T}
Let $H$ be a locally compact group and $K$ be an LCA group also $\tau:H\to Aut(K)$ be a continuous homomorphism and let $\delta:H\to(0,\infty)$ be the positive continuous homomorphism satisfying $dk=\delta(h)dk^h$. The semi-direct product $G_{\widehat{\tau}}=H\ltimes_{\widehat{\tau}}\widehat{K}$ is a locally compact group with the left Haar measure $d\mu_{G_{\widehat{\tau}}}(h,\omega)=\delta(h)^{-1}dhd\omega$.
\end{theorem}
\begin{proof}
Continuity of the homomorphism $\widehat{\tau}:H\to Aut(\widehat{K})$ given in (\ref{AAA}) guaranteed by Theorem 26.9 of \cite{HR1}.
Hence, the semi-direct product $G_{\widehat{\tau}}=H\ltimes_{\widehat{\tau}}\widehat{K}$ is a locally compact group.
We also claim that the Plancherel measure $d\omega$ on $\widehat{K}$ for all $h\in H$ satisfies
\begin{equation}\label{AA2}
d\omega_h=\delta(h)d\omega.
\end{equation}
Let $h\in H$ and $v\in L^1(K)$. Using (\ref{t}) we have $v\circ \tau_h\in L^1(K)$ with $\|v\circ \tau_h\|_{L^1(K)}=\delta(h)\|v\|_{L^1(K)}$.
Thus, for all $\omega\in\widehat{K}$ we achieve
\begin{align*}
\widehat{v\circ\tau_h}(\omega)
&=\int_Kv\circ\tau_h(k)\overline{\omega(k)}dk
\\&=\int_Kv(k^h)\overline{\omega(k)}dk
\\&=\int_Kv(k)\overline{\omega_h(k)}dk^{h^{-1}}
=\delta(h)\int_Kv(k)\overline{\omega_h(k)}dk=\delta(h)\widehat{v}(\omega_h).
\end{align*}

Now let $v\in L^1(K)\cap L^2(K)$. Due to the Plancherel theorem (Theorem 4.25 of \cite{FollH}) and also preceding calculation, for all $h\in H$ we get
\begin{align*}\label{}
\int_{\widehat{K}}|\widehat{v}(\omega)|^2d\omega_h
&=\int_{\widehat{K}}|\widehat{v}(\omega_{h^{-1}})|^2d\omega
\\&=\delta(h)^2\int_{\widehat{K}}|\widehat{v\circ\tau_{h^{-1}}}(\omega)|^2d\omega
\\&=\delta(h)^2\int_{{K}}|{v\circ\tau_{h^{-1}}}(k)|^2dk
\\&=\delta(h)^2\int_{{K}}|v(k)|^2dk^h
=\delta(h)\int_{{K}}|v(k)|^2dk=\int_{\widehat{K}}|\widehat{v}(\omega)|^2\delta(h)d\omega,
\end{align*}
which implies (\ref{AA2}). Therefore, $d\mu_{G_{\widehat{\tau}}}(h,\omega)=\delta(h)^{-1}dhd\omega$ is a left Haar measure for $G_{\widehat{\tau}}=H\ltimes_{\widehat{\tau}}\widehat{K}$.
\end{proof}


\begin{remark}
Due to (\ref{AAA}) for all $k\in K$, $\omega\in\widehat{K}$ and $h,t\in H$ we have
\begin{equation}
k^{ht}=(k^t)^h,\hspace{1cm}\omega_{ht}=(\omega_t)_h.
\end{equation}
\end{remark}


Now we are in the position to introduce the $\tau\times\widehat{\tau}$-time frequency group. Define $\tau^{\times}=\tau\times\widehat{\tau}:H\to Aut(K\times\widehat{K})$ via $h\mapsto\tau^{\times}_h$ given by
\begin{equation}\label{2.0}
\tau_h^\times(k,\omega):=\left(\tau_h(k),\widehat{\tau}_h(\omega)\right)=(k^h,\omega_h),
\end{equation}
for all $(k,\omega)\in K\times\widehat{K}$. Then, for all $h\in H$ we have $\tau_h^\times\in Aut(K\times\widehat{K})$. Because for all $(k,\omega),(k',\omega')\in K\times\widehat{K}$ we have
\begin{align*}
\tau^\times_{h}\left((k,\omega)(k',\omega')\right)
&=\tau^\times_{h}(kk',\omega\omega')
\\&=\left((kk')^h,(\omega\omega')_h\right)
\\&=\left(k^hk'^h,\omega_h\omega'_h\right)
\\&=(k^h,\omega_h)(k'^h,\omega'_h)=\tau^\times_{h}(k,\omega)\tau^\times_{h}(k',\omega').
\end{align*}
Also $\tau^\times=\tau\times\widehat{\tau}:H\to Aut(K\times\widehat{K})$ defined by $h\mapsto\tau^\times_{h}$ is a homomorphism, because for all $h,t\in H$ and all $(k,\omega)\in K\times\widehat{K}$ we have
\begin{align*}
\tau^\times_{ht}(k,\omega)
&=(k^{ht},\omega_{ht})
\\&=\left((k^{t})^h,(\omega_{t})_h\right)
\\&=\tau^\times_{h}(k^{t},\omega_{t})=\tau^\times_{h}\tau^\times_{t}(k,\omega).
\end{align*}


In the following proposition we show that $G_{\tau\times\widehat{\tau}}=H\ltimes_{\tau\times\widehat{\tau}}(K\times\widehat{K})$ is a locally compact group.


\begin{proposition}\label{TT}
{\it Let $H$ be a locally compact group and $K$ be an LCA group also $\tau:H\to Aut(K)$ be a continuous homomorphism and let $\delta:H\to(0,\infty)$ be the positive continuous homomorphism satisfying $dk=\delta(h)dk^h$. The semi-direct product  $G_{\tau\times\widehat{\tau}}=H\ltimes_{\tau\times\widehat{\tau}}(K\times\widehat{K})$ is a locally compact group with the left Haar measure \begin{equation}
d\mu_{G_{\tau\times\widehat{\tau}}}(h,k,\omega)=dhdkd\omega.
\end{equation}
}\end{proposition}
\begin{proof}
Continuity of the homomorphism $\tau\times\widehat{\tau}:H\to Aut(K\times\widehat{K})$ given in (\ref{2.0}) guaranteed by Theorem 26.9 of \cite{HR1}. Thus, the semi-direct product  $G_{\tau\times\widehat{\tau}}=H\ltimes_{\tau\times\widehat{\tau}}(K\times\widehat{K})$ is a locally compact group.
Due to (\ref{t}), (\ref{AA2}) and also (\ref{2.0}), for all $h\in H$ we have
\begin{align*}
d\tau^{\times}_h(k,\omega)&=d\left(k^h,\omega_h\right)
\\&=dk^hd\omega_h
\\&=\delta(h)^{-1}dk\delta(h)d\omega=dkd\omega=d(k,\omega),
\end{align*}
which implies that $G_{\tau\times\widehat{\tau}}$ is a locally compact group with the left Haar measure $d\mu_{G_{\tau\times\widehat{\tau}}}(h,k,\omega)=dhdkd\omega$.
\end{proof}


We call the semi-direct product $G_{\tau\times\widehat{\tau}}$ as the $\tau\times\widehat{\tau}$-time frequency group associated to $G_\tau$. According to (\ref{2.0}) for each $(h,k,\omega),(h',k',\omega')\in G_{\tau\times\widehat{\tau}}$ we have
\begin{align*}
(h,k,\omega)\ltimes_{\tau\times\widehat{\tau}}(h',k',\omega')
&=\left(hh',(k,\omega)\tau_h^{\times}(k',\omega')\right)
\\&=\left(hh',(k,\omega)(\tau_h(k'),\omega'_h)\right)=(hh',k+k'^h,\omega\omega'_h).
\end{align*}


Let $u\in L^2(K)$ be a window function and $f\in L^2(G_\tau)$. The {\it $\tau\times\widehat{\tau}$-continuous Gabor transform} of $f$ with respect to the window function $u$ is define by
\begin{equation}\label{CGT1}
\mathcal{V}_uf(h,k,\omega):=\delta(h)^{1/2}V_{u}f_h(k,\omega)=\delta(h)^{1/2}\langle f_h,\varrho(k,\omega)u\rangle_{L^2(K)}.
\end{equation}


In the following theorem we prove a Plancherel formula for the $\tau\times\widehat{\tau}$-continuous Gabor transform defined in (\ref{CGT1}).


\begin{theorem}\label{P1}
Let $H$ be a locally compact group and $K$ be an LCA group also $\tau:H\to Aut(K)$ be a continuous homomorphism and let $u\in L^2(K)$ be a window function. The $\tau\times\widehat{\tau}$-continuous Gabor transform $\mathcal{V}_u:L^2(G_\tau)\to L^2(G_{\tau\times\widehat{\tau}})$ is a multiple of an isometric transform which maps $L^2(G_\tau)$ onto a closed subspace of $L^2(G_{\tau\times\widehat{\tau}})$.
\end{theorem}
\begin{proof}
Let $u\in L^2(K)$ be a window function and also $f\in L^2(G_\tau)$. Using Fubini's Theorem and also Plancherel formula (\ref{PL0}) we have
\begin{align*}
\|\mathcal{V}_uf\|_{L^2(G_{\tau\times\widehat{\tau}})}^2
&=\int_{G_{\tau\times\widehat{\tau}}}|\mathcal{V}_uf(h,k,\omega)|^2d\mu_{G_{\tau\times\widehat{\tau}}}(h,k,\omega)
\\&=\int_H\int_K\int_{\widehat{K}}|\mathcal{V}_uf(h,k,\omega)|^2dhdkd\omega
\\&=\int_H\left(\int_{K\times\widehat{K}}|\langle f_h,\varrho(k,\omega)u\rangle_{L^2(K)}|^2dkd\omega\right)\delta(h)dh
\\&=\|u\|_{L^2(K)}^2\int_H\|f_h\|_{L^2(K)}^2\delta(h)dh=\|u\|_{L^2(K)}^2\|f\|_{L^2(G_\tau)}^2.
\end{align*}
Therefore, $\|u\|_{L^2(K)}^{-2}\mathcal{V}_u:L^2(G_\tau)\to L^2(G_\tau\times G_{\widehat{\tau}})$ is an isometric transform with a closed range in $L^2(G_{\tau}\times G_{\widehat{\tau}})$.
\end{proof}


\begin{corollary}
{\it The $\tau\times\widehat{\tau}$-continuous Gabor transform defined in (\ref{CGT1}), for all $f,g\in L^2(G_\tau)$ and window functions $u,v\in L^2(K)$ satisfies the following orthogonality relation;
\begin{equation}
\langle\mathcal{V}_uf,\mathcal{V}_{v}g\rangle_{L^2(G_{\tau\times\widehat{\tau}})}=\langle v,u\rangle_{L^2(K)}\langle f,g\rangle_{L^2(G_\tau)}.
\end{equation}
}\end{corollary}


The $\tau\times\widehat{\tau}$-continuous Gabor transform (\ref{CGT1}) satisfies the following inversion formula.


\begin{proposition}\label{I1}
{\it Let $H$ be a locally compact group and $K$ be an LCA group also let $\tau:H\to Aut(K)$ be a continuous homomorphism and $u\in L^2(K)$ with $\widehat{u}\in L^1(\widehat{K})$. Every $f\in L^2(G_\tau)$ with $\widehat{f_h}\in L^1(\widehat{K})$ for a.e. $h\in H$, satisfies the following reconstruction formula; }
\begin{equation}\label{I10}
f(h,k)=\delta(h)^{-1/2}\|u\|^{-2}_{L^2(K)}\int_{K\times\widehat{K}}\mathcal{V}_uf(h,s,\omega)[\varrho(s,\omega)u](k)dsd\omega.
\end{equation}
\end{proposition}
\begin{proof}
Using (\ref{PI}) for a.e. $h\in H$ we have
\begin{align*}
f_h(k)&=\|u\|^{-2}_{L^2(K)}\int_{K\times\widehat{K}}V_{u}f_h(s,\omega)[\varrho(s,\omega)u](k)dsd\omega
\\&=\delta(h)^{-1/2}\|u\|^{-2}_{L^2(K)}\int_{K\times\widehat{K}}\mathcal{V}_{u}f(h,s,\omega)[\varrho(s,\omega)u](k)dsd\omega.
\end{align*}

\end{proof}

We can also define the generalized form of the $\tau\times\widehat{\tau}$-continuous Gabor transform. Let $u\in L^2(K)$ be a window function and $f\in L^2(G_\tau)$. The {\it generalized $\tau\times\widehat{\tau}$-continuous Gabor transform} of $f$ with respect to the window function $u$ is define by
\begin{equation}\label{GCGT1}
\mathcal{V}^\dagger_uf(h,k,\omega):=\delta(h)^{1/2}V_{u}f_h(k^h,\omega_h)=\delta(h)^{1/2}\langle f_h,\varrho(k^h,\omega_h)u\rangle_{L^2(K)}.
\end{equation}

The generalized $\tau\times\widehat{\tau}$-continuous Gabor transform given in (\ref{GCGT1}) satisfies the following Plancherel Theorem.


\begin{theorem}\label{P2}
Let $H$ be a locally compact group and $K$ be an LCA group also $\tau:H\to Aut(K)$ be a continuous homomorphism and let $u\in L^2(K)$ be a window function. The generalized $\tau\times\widehat{\tau}$-continuous Gabor transform $\mathcal{V}^\dagger_u:L^2(G_\tau)\to L^2(G_{\tau\times\widehat{\tau}})$ is a multiple of an isometric transform which maps $L^2(G_\tau)$ onto a closed subspace of $L^2(G_{\tau\times\widehat{\tau}})$.
\end{theorem}
\begin{proof}
Let $u\in L^2(K)$ be a window function and also $f\in L^2(G_\tau)$. Using Fubini's Theorem, Plancherel formula (\ref{PL0}) and also (\ref{t}), (\ref{AA2}) we have
\begin{align*}
\|\mathcal{V}^\dagger_uf\|_{L^2(G_{\tau\times\widehat{\tau}})}^2
&=\int_{G_{\tau\times\widehat{\tau}}}|\mathcal{V}^\dagger_uf(h,k,\omega)|^2d\mu_{G_{\tau\times\widehat{\tau}}}(h,k,\omega)
\\&=\int_H\int_K\int_{\widehat{K}}|\mathcal{V}^\dagger_uf(h,k,\omega)|^2dhdkd\omega
\\&=\int_H\left(\int_{K\times\widehat{K}}|\langle f_h,\varrho(k^h,\omega_h)u\rangle_{L^2(K)}|^2dkd\omega\right)\delta(h)dh
\\&=\int_H\left(\int_{K\times\widehat{K}}|\langle f_h,\varrho(k,\omega)u\rangle_{L^2(K)}|^2dk^{h^{-1}}d\omega_{h^{-1}}\right)\delta(h)dh
\\&=\int_H\left(\int_{K\times\widehat{K}}|\langle f_h,\varrho(k,\omega)u\rangle_{L^2(K)}|^2dkd\omega\right)\delta(h)dh
\\&=\|u\|_{L^2(K)}^2\int_H\|f_h\|_{L^2(K)}^2\delta(h)dh=\|u\|_{L^2(K)}^2\|f\|_{L^2(G_\tau)}^2.
\end{align*}
Thus, $\|u\|_{L^2(K)}^{-2}\mathcal{V}^\dagger_u:L^2(G_\tau)\to L^2(G_{\tau\times\widehat{\tau}})$ is an isometric transform with a closed range in $L^2(G_{\tau\times\widehat{\tau}})$.
\end{proof}


\begin{corollary}
{\it The generalized $\tau\times\widehat{\tau}$-continuous Gabor transform defined in (\ref{GCGT1}), for all $f,g\in L^2(G_\tau)$ and window functions $u,v\in L^2(K)$ satisfies the following orthogonality relation;
\begin{equation}
\langle\mathcal{V}_u^\dagger f,\mathcal{V}_{v}^\dagger g\rangle_{L^2(G_{\tau\times\widehat{\tau}})}=\langle v,u\rangle_{L^2(K)}\langle f,g\rangle_{L^2(G_\tau)}.
\end{equation}
}\end{corollary}


In the next proposition we prove an inversion formula for the generalized $\tau\times\widehat{\tau}$-continuous gabor transform given in (\ref{GCGT1}).

\begin{proposition}\label{I2}
{\it Let $H$ be a locally compact group and $K$ be an LCA group also let $\tau:H\to Aut(K)$ be a continuous homomorphism and $u\in L^2(K)$ with $\widehat{u}\in L^1(\widehat{K})$. Every $f\in L^2(G_\tau)$ with $\widehat{f_h}\in L^1(\widehat{K})$ for a.e. $h\in H$, satisfies the following reconstruction formula;
\begin{equation}\label{I20}
f(h,k)=\int_{K\times\widehat{K}}\mathcal{V}_u^\dagger f(h,s,\omega)[\varrho(s^h,\omega_h)u](k)dsd\omega
\end{equation}
}\end{proposition}
\begin{proof}
Using (\ref{PI}) for a.e. $h\in H$ we have
\begin{align*}
f_h(k)&=\int_{K\times\widehat{K}}V_{u}f_h(s,\omega)[\varrho(s,\omega)u](k)dsd\omega
\\&=\int_{K}\left(\int_{\widehat{K}}V_{u}f_h(s,\omega_h)[\varrho(s,\omega_h)u](k)d\omega_h\right)ds
\\&=\delta(h)\int_{\widehat{K}}\left(\int_KV_{u}f_h(s,\omega_h)[\varrho(s,\omega_h)u](k)ds\right)d\omega
\\&=\delta(h)\int_{\widehat{K}}\left(\int_KV_{u}f_h(s^h,\omega_h)[\varrho(s^h,\omega_h)u](k)ds^h\right)d\omega
=\int_{K\times\widehat{K}}\mathcal{V}_u^\dagger f(h,s,\omega)[\varrho(s^h,\omega_h)u](k)dsd\omega.
\end{align*}
\end{proof}


\begin{remark}
Similar Gabor transforms with respect to a window function $u\in L^2(K)$ as we defined in (\ref{CGT1}) and (\ref{GCGT1}) can be also defined.
Let transforms $A_u$ and $B_u$ for $f\in L^2(G_\tau)$ be given by
\begin{equation}\label{AB1}
A_uf(h,k,\omega)=V_uf_h(k^h,\omega)\hspace{1cm}B_uf(h,k,\omega)=\delta(h)V_uf_h(k,\omega_h).
\end{equation}
It can be checked that transforms given in (\ref{AB1}) satisfy the Plancherel theorem and the following inversion formulas;
\begin{equation}
f(h,k)=\delta(h)^{-1}\int_{K\times\widehat{K}}A_uf(h,k,\omega)[\varrho(s^h,\omega)](k)dsd\omega,
\end{equation}
\begin{equation}
f(h,k)=\int_{K\times\widehat{K}}B_uf(h,k,\omega)[\varrho(s,\omega_h)](k)dsd\omega.
\end{equation}
\end{remark}

\section{{\bf $\tau\otimes\widehat{\tau}$-continuous Gabor transform}}

In this section we introduce another Gabor transform which we call it the $\tau\otimes\widehat{\tau}$-continuous Gabor transform. In the $\tau\otimes\widehat{\tau}$-Gabor theory we can choose elements of $L^2(G_\tau)$ as window functions.

Again let $H$ be a locally compact group and $K$ be an LCA group also let $\tau:H\to Aut(K)$ be a continuous homomorphism.
Define $\tau^\otimes=\tau\otimes\widehat{\tau}:H\times H\to Aut(K\times\widehat{K})$ via $(h,t)\mapsto\tau^\otimes_{(h,t)}$ given by
\begin{equation}\label{T2}
\tau^\otimes_{(h,t)}(k,\omega):=(\tau_h(k),\widehat{\tau}_t(\omega))=(k^h,\omega_t),
\end{equation}
for all $(k,\omega)\in K\times\widehat{K}$. Then, for all $(h,t)\in H\times H$ we get $\tau^\otimes_{(h,t)}\in Aut(K\times\widehat{K})$. Because for all $(k,\omega),(k',\omega')\in K\times\widehat{K}$ we have
\begin{align*}
\tau^\otimes_{(h,t)}\left((k,\omega)(k',\omega')\right)
&=\tau^\otimes_{(h,t)}(k+k',\omega\omega')
\\&=\left((k+k')^h,(\omega\omega')_t\right)
\\&=\left(k^h+k'^h,\omega_t\omega'_t\right)
\\&=(k^h,\omega_t)(k'^h,\omega'_t)=\tau^\otimes_{(h,t)}(k,\omega)\tau^\otimes_{(h,t)}(k',\omega').
\end{align*}
As well as $\tau^\otimes=\tau\otimes\widehat{\tau}:H\times H\to Aut(K\times\widehat{K})$ defined by $(h,t)\mapsto\tau^\otimes_{(h,t)}$ is a homomorphism, because for all $(h,t),(h',t')\in H\times H$ and also all $(k,\omega)\in K\times\widehat{K}$ we have
\begin{align*}
\tau^\otimes_{(h,t)(h',t')}(k,\omega)
&=\tau^\otimes_{(hh',tt')}(k,\omega)
\\&=(k^{hh'},\omega_{tt'})
\\&=\left((k^{h'})^h,(\omega_{t'})_t\right)
\\&=\tau^\otimes_{(h,t)}(k^{h'},\omega_{t'})=\tau^\otimes_{(h,t)}\tau^\otimes_{(h',t')}(k,\omega).
\end{align*}


Hence, we can prove the following interesting theorem.


\begin{theorem}\label{TTT}
Let $H$ be a locally compact group and $K$ be an LCA group also let $\tau:H\to Aut(K)$ be a continuous homomorphism. The semi-direct product $G_{\tau\otimes\widehat{\tau}}=\left(H\times H\right)\ltimes_{\tau\otimes\widehat{\tau}}\left(K\times\widehat{K}\right)$ is a locally compact group with the left Haar measure
\begin{equation}
d\mu_{G_{\tau\otimes\widehat{\tau}}}(h,t,k,\omega)=\delta(h)\delta(t)^{-1}dhdtdkd\omega,
\end{equation}
and also $\Phi:G_\tau\times G_{\widehat{\tau}}\to G_{\tau\otimes\widehat{\tau}}$ given by
\begin{equation}
(h,k,t,\omega)\mapsto\Phi(h,k,t,\omega):=(h,t,k,\omega)
\end{equation}
is a topological group isomorphism.
\end{theorem}
\begin{proof}
Using Theorem 26.9 of \cite{HR1}, homomorphism $\tau\otimes\widehat{\tau}:H\times H\to Aut(K\times\widehat{K})$ given in (\ref{T2}) is continuous. Therefore, $G_{\tau\otimes\widehat{\tau}}=\left(H\times H\right)\ltimes_{\tau\otimes\widehat{\tau}}\left(K\times\widehat{K}\right)$ is a locally compact group. Also, $d\mu_{G_{\tau\otimes\widehat{\tau}}}(h,t,k,\omega)=\delta(h)\delta(t)^{-1}dhdtdkd\omega$ is a left Haar measure for $G_{\tau\otimes\widehat{\tau}}$. Indeed, due to (\ref{t}) and (\ref{AA2}) for all $(h,t)\in H\times H$ we have
\begin{align*}
d\tau^\otimes_{(h,t)}(k,\omega)&=d(k^h,\omega_t)
\\&=dk^hd\omega_t
\\&=\delta(h)^{-1}dk\delta(t)d\omega
=\delta(h)^{-1}\delta(t)d(k,\omega).
\end{align*}
The $\tau\otimes\widehat{\tau}$-group law for all $(h,t,k,\omega),(h',t',k',\omega')\in G_{\tau\otimes\widehat{\tau}}$ is
\begin{align*}
(h,t,k,\omega)\ltimes_{\tau\otimes\widehat{\tau}}(h',t',k',\omega')
&=\left((hh',tt'),(k,\omega)\tau^\otimes_{(h,t)}(k',\omega')\right)
\\&=\left((hh',tt'),(k,\omega)(k'^h,\omega'_t)\right)=(hh',tt',k+k'^h,\omega\omega_t').
\end{align*}
It is clear that $\Phi:G_\tau\times G_{\widehat{\tau}}\to G_{\tau\otimes\widehat{\tau}}$ is a homeomorphism.
It is also a group homomorphism, because for all $(h,k,t,\omega),(h',k',t',\omega')$ in $G_\tau\times G_{\widehat{\tau}}$ we get
\begin{align*}
\Phi[(h,k,t,\omega)(h',k',t',\omega')]
&=\Phi[(h,k)\ltimes_\tau(h',k'),(t,\omega)\ltimes_{\widehat{\tau}}(t',\omega')]
\\&=\Phi[(hh',k+k'^h),(tt',\omega\omega_t')]
\\&=(hh',tt',k+k'^h,\omega\omega_t')
=(h,t,k,\omega)\ltimes_{\tau\otimes\widehat{\tau}}(h',t',k',\omega').
\end{align*}
\end{proof}

We call the semi-direct product $G_{\tau\otimes\widehat{\tau}}$ as the $\tau\otimes\widehat{\tau}$-time frequency group associated to $G_\tau$ which is precisely $G_\tau\times G_{\widehat{\tau}}$. Thus, form now on we use the locally compact group $G_\tau\times G_{\widehat{\tau}}$ instead of the semi-direct product $G_{\tau\otimes\widehat{\tau}}$.


Let $g\in L^2(G_\tau)$ be a window function and $f\in L^2(G_\tau)$. The {\it $\tau\otimes\widehat{\tau}$-continuous Gabor transform} of $f$ with respect to the window function $g$ is defined by
\begin{equation}\label{CGT2}
\mathcal{G}_gf(h,k,t,\omega):=\delta(t)V_{g_h}f_t(k,\omega)=\delta(t)\langle f_t,\varrho(k,\omega)g_h\rangle_{L^2(K)}.
\end{equation}


The $\tau\otimes\widehat{\tau}$-continuous Gabor transform given in (\ref{CGT2}) satisfies the following Plancherel Theorem.


\begin{theorem}\label{P3}
Let $H$ be a locally compact group, $K$ be an LCA group and $\tau:H\to Aut(K)$ be a continuous homomorphism also $G_\tau=H\ltimes_\tau K$ and let $g\in L^2(G_\tau)$ be a window function. The continuous Gabor transform $\mathcal{G}_g:L^2(G_\tau)\to L^2(G_\tau\times G_{\widehat{\tau}})$ is a multiple of an isometric transform which maps $L^2(G_\tau)$ onto a closed subspace of $L^2(G_{\tau}\times G_{\widehat{\tau}})$.
\end{theorem}
\begin{proof}
Let $g\in L^2(G_\tau)$ be a window function and also let $f\in L^2(G_\tau)$.
\begin{align*}
\|\mathcal{G}_gf\|_{L^2(G_\tau\times G_{\widehat{\tau}})}^2
&=\int_{G_\tau\times G_{\widehat{\tau}}}|\mathcal{G}_gf(h,k,t,\omega)|^2d\mu_{G_\tau\times G_{\widehat{\tau}}}(h,k,t,\omega)
\\&=\int_{G_\tau}\int_{G_{\widehat{\tau}}}|\mathcal{G}_gf(h,k,t,\omega)|^2d\mu_{G_\tau}(h,k) d\mu_{G_{\widehat{\tau}}}(t,\omega)
\\&=\int_H\int_K\int_H\int_{\widehat{K}}|\mathcal{G}_gf(h,k,t,\omega)|^2\delta(h)dhdk\delta(t)^{-1}dtd\omega
\\&=\int_H\int_H\left(\int_{K\times\widehat{K}}|\langle f_t,\varrho(k,\omega)g_h\rangle_{L^2(K)}|^2dkd\omega\right) \delta(h)dh\delta(t)dt
\\&=\int_H\int_H\|f_t\|_{L^2(K)}^2\|g_h\|_{L^2(K)}^2\delta(h)dh\delta(t)dt=\|f\|_{L^2(G_\tau)}^2\|g\|_{L^2(G_\tau)}^2
\end{align*}
Thus, $\|g\|_{L^2(G_\tau)}^{-2}\mathcal{G}_g:L^2(G_\tau)\to L^2(G_\tau\times G_{\widehat{\tau}})$ is an isometric transform with a closed range in $L^2(G_{\tau}\times G_{\widehat{\tau}})$.
\end{proof}


\begin{corollary}
{\it The $\tau\times\widehat{\tau}$-continuous Gabor transform defined in (\ref{CGT2}), for all $f,f'\in L^2(G_\tau)$ and window functions $g,g'\in L^2(G_\tau)$ satisfies the following orthogonality relation;
\begin{equation}
\langle\mathcal{G}_gf,\mathcal{G}_{g'}f'\rangle_{L^2(G_\tau\times G_{\widehat{\tau}})}=\langle g',g\rangle_{L^2(G_\tau)}\langle f,f'\rangle_{L^2(G_\tau)}.
\end{equation}
}\end{corollary}


In the following proposition we also prove an inversion formula.


\begin{proposition}\label{I3}
{\it Let $H$ be a locally compact group and $K$ be an LCA group also let $\tau:H\to Aut(K)$ be a continuous homomorphism. Every $f,g\in L^2(G_\tau)$ with $\widehat{f_h},\widehat{g_h}\in L^1(\widehat{K})$ for a.e. $h\in H$, satisfy the following reconstruction formula;
\begin{equation}
f(t,k)=\langle g_h,g_h\rangle^{-1}_{L^2(K)}\delta(t)^{-1}\int_{K\times\widehat{K}}\mathcal{G}_gf(h,s,t,\omega)[\varrho(s,\omega)g_h](k)dsd\omega,
\end{equation}
for a.e. $h,t\in H$ and $k\in K$. In particular, for a.e. $h\in H$ we have
\begin{equation}\label{I30}
f(h,k)=\langle g_h,g_h\rangle^{-1}_{L^2(K)}\delta(h)^{-1}\int_{K\times\widehat{K}}\mathcal{G}_gf(h,s,h,\omega)[\varrho(s,\omega)g_h](k)dsd\omega.
\end{equation}
}\end{proposition}
\begin{proof}
Using (\ref{PI}) for a.e. $h,t\in H$ we have
\begin{align*}
f_t(k)&=\langle g_h,g_h\rangle_{L^2(K)}^{-1}\int_{K\times\widehat{K}}V_{g_h}f_t(s,\omega)[\varrho(s,\omega)g_h](k)dsd\omega
\\&=\langle g_h,g_h\rangle^{-1}_{L^2(K)}\delta(t)^{-1}\int_{K\times\widehat{K}}\mathcal{G}_gf(h,s,t,\omega)[\varrho(s,\omega)g_h](k)dsd\omega.
\end{align*}
\end{proof}


Let $g\in L^2(G_\tau)$ be a window function and $f\in L^2(G_\tau)$.
The {\it generalized $\tau\otimes\widehat{\tau}$-continuous Gabor transform} of $f$ with respect to the window function $g$ is defined by
\begin{equation}\label{GCGT2}
\mathcal{G}_g^\dagger f(h,k,t,\omega):=\delta(h)^{-1/2}\delta(t)^{3/2}V_{g_h}f_t(k^h,\omega_t)
=\delta(h)^{-1/2}\delta(t)^{3/2}\langle f_t,\varrho(k^h,\omega_t)g_h\rangle_{L^2(K)}.
\end{equation}

In the next theorem, a Plancherel formula for the generalized
$\tau\otimes\widehat{\tau}$-continuous Gabor transform defined in (\ref{GCGT2}) proved.


\begin{theorem}\label{P4}
Let $H$ be a locally compact group, $K$ be an LCA group and $\tau:H\to Aut(K)$ be a continuous homomorphism also $G_\tau=H\ltimes_\tau K$ and also let $g\in L^2(G_\tau)$ be a window function. The generalized continuous Gabor transform $\mathcal{G}_g^\dagger:L^2(G_\tau)\to L^2(G_\tau\times G_{\widehat{\tau}})$ is a multiple of an isometric transform which maps $L^2(G_\tau)$ onto a closed subspace of $L^2(G_{\tau}\times G_{\widehat{\tau}})$.
\end{theorem}
\begin{proof}
Let $g\in L^2(G_\tau)$ be a window function and also let $f\in L^2(G_\tau)$. Using Fubini's theorem and also Theorem we achieve
\begin{align*}
\|\mathcal{G}_gf\|_{L^2(G_\tau\times G_{\widehat{\tau}})}^2
&=\int_{G_\tau\times G_{\widehat{\tau}}}|\mathcal{G}_g^\dagger f(h,k,t,\omega)|^2d\mu_{G_\tau\times G_{\widehat{\tau}}}(h,k,t,\omega)
\\&=\int_{G_\tau}\int_{G_{\widehat{\tau}}}|\mathcal{G}_g^\dagger f(h,k,t,\omega)|^2d\mu_{G_\tau}(h,k) d\mu_{G_{\widehat{\tau}}}(t,\omega)
\\&=\int_H\int_K\int_H\int_{\widehat{K}}|\mathcal{G}_g^\dagger f(h,k,t,\omega)|^2\delta(h)dhdk\delta(t)^{-1}dtd\omega
\\&=\int_H\int_H\left(\int_{K\times\widehat{K}}|\langle f_t,\varrho(k^h,\omega_t)g_h\rangle_{L^2(K)}|^2dkd\omega\right) dh\delta(t)^{2}dt
\\&=\int_H\int_H\left(\int_{K\times\widehat{K}}|\langle f_t,\varrho(k,\omega)g_h\rangle_{L^2(K)}|^2dk^{h^{-1}}d\omega_{t^{-1}}\right) dh\delta(t)^{2}dt
\\&=\int_H\int_H\left(\int_{K\times\widehat{K}}|\langle f_t,\varrho(k,\omega)g_h\rangle_{L^2(K)}|^2dkd\omega\right) \delta(h)dh\delta(t)dt
\\&=\int_H\int_H\|f_t\|_{L^2(K)}^2\|g_h\|_{L^2(K)}^2\delta(h)dh\delta(t)dt=\|f\|_{L^2(G_\tau)}^2\|g\|_{L^2(G_\tau)}^2
\end{align*}
Thus, $\|g\|_{L^2(G_\tau)}^{-2}\mathcal{G}_g^\dagger:L^2(G_\tau)\to L^2(G_\tau\times G_{\widehat{\tau}})$ is an isometric transform with a closed range in $L^2(G_{\tau}\times G_{\widehat{\tau}})$.
\end{proof}


\begin{corollary}
{\it The $\tau\times\widehat{\tau}$-continuous Gabor transform defined in (\ref{GCGT2}), for all $f,f'\in L^2(G_\tau)$ and window functions $g,g'\in L^2(G_\tau)$ satisfies the following orthogonality relation;
\begin{equation}
\langle\mathcal{G}_g^\dagger f,\mathcal{G}_{g'}^\dagger f'\rangle_{L^2(G_\tau\times G_{\widehat{\tau}})}=\langle g',g\rangle_{L^2(G_\tau)}\langle f,f'\rangle_{L^2(G_\tau)}.
\end{equation}
}\end{corollary}


Also, the generalized $\tau\otimes\widehat{\tau}$-continuous Gabor transform satisfies the following inversion formula.


\begin{proposition}\label{I40}
{\it Let $H$ be a locally compact group and $K$ be an LCA group also let $\tau:H\to Aut(K)$ be a continuous homomorphism. Every $f,g\in L^2(G_\tau)$ with $\widehat{f_h},\widehat{g_h}\in L^1(\widehat{K})$ for a.e. $h,t\in H$, satisfy the following reconstruction formula;
\begin{equation}
f(t,k)=\langle g_h,g_h\rangle^{-1}_{L^2(K)}\delta(h)^{-1/2}\delta(t)^{-1/2}\int_{K\times\widehat{K}}\mathcal{G}_g^\dagger f(h,s,t,\omega)[\varrho(k^h,\omega_t)g_h](k)dsd\omega,
\end{equation}
for a.e. $h,t\in H$ and $k\in K$. In particular for a.e. $h\in H$ we have
\begin{equation}
f(h,k)=\langle g_h,g_h\rangle^{-1}_{L^2(K)}\delta(h)^{-1}\int_{K\times\widehat{K}}\mathcal{G}_g^\dagger f(h,s,h,\omega)[\varrho(s^h,\omega_h)g_h](k)dsd\omega.
\end{equation}
}\end{proposition}
\begin{proof}
Using (\ref{PI}) for a.e. $h,t\in H$ we have
\begin{align*}
f_t(k)&=\langle g_h,g_h\rangle^{-1}_{L^2(K)}\int_{K\times\widehat{K}}V_{g_h}f_t(s,\omega)[\varrho(s,\omega)g_h](k)dsd\omega
\\&=\langle g_h,g_h\rangle^{-1}_{L^2(K)}\int_K\int_{\widehat{K}}V_{g_h}f_t(s,\omega)[\varrho(s,\omega)g_h](k)d\omega ds
\\&=\langle g_h,g_h\rangle^{-1}_{L^2(K)}\int_K\left(\int_{\widehat{K}}V_{g_h}f_t(s,\omega_t)[\varrho(s,\omega_t)g_h](k)d\omega_t\right) ds
\\&=\langle g_h,g_h\rangle^{-1}_{L^2(K)}\delta(t)\int_{\widehat{K}}\left(\int_KV_{g_h}f_t(s^h,\omega_t)[\varrho(s^h,\omega_t)g_h](k)ds^h\right) d\omega
\\&=\langle g_h,g_h\rangle^{-1}_{L^2(K)}\delta(h)^{-1}\delta(t)\int_{\widehat{K}}\left(\int_KV_{g_h}f_t(s^h,\omega_t)[\varrho(s^h,\omega_t)g_h](k)ds\right) d\omega
\\&=\langle g_h,g_h\rangle^{-1}_{L^2(K)}\delta(h)^{-1/2}\delta(t)^{-1/2}\int_{K\times\widehat{K}}\mathcal{G}_gf(h,s,t,\omega)[\varrho(s^h,\omega_t)g_h](k)dsd\omega.
\end{align*}
\end{proof}

\section{\bf{Examples and applications}}

\subsection{The Affine group $\mathbf{a}\mathbf{x}+\mathbf{b}$}
Let $H=\mathbb{R}^*_+=(0,+\infty)$ and $K=\mathbb{R}$. The affine group $a{\bf x}+b$ is the semi direct product $H\ltimes_\tau K$ with respect to the homomorphism $\tau:H\to Aut(K)$ given by $a\mapsto \tau_a$,
where $\tau_a(x)=ax$ for all $x\in\mathbb{R}$. Hence, the underlying manifold of the affine group is $(0,\infty)\times\mathbb{R}$ and also the group law is
\begin{equation}
(a,x)\ltimes_\tau(a',x')=(aa',x+ax').
\end{equation}
The continuous homomorphism $\delta:H\to(0,\infty)$ is given by $\delta(a)=a^{-1}$ and so that the left Haar measure is in fact $d\mu_{G_\tau}(a,x)=a^{-2}dadx$. Due to Theorem 4.5 of \cite{FollH} we can identify $\widehat{\mathbb{R}}$ with $\mathbb{R}$ via $\omega(x)=\langle x,\omega\rangle=e^{2\pi i\omega x}$ for each $\omega\in\widehat{\mathbb{R}}$ and so we can consider the continuous homomorphism $\widehat{\tau}:H\to Aut(\widehat{K})$ given by $a\mapsto\widehat{\tau}_a$ via
\begin{align*}\label{500}
\langle x,\omega_a\rangle&=\langle x,\widehat{\tau}_a(\omega)\rangle
\\&=\langle\tau_{a^{-1}}(x),\omega\rangle
=\langle a^{-1}x,\omega\rangle=e^{2\pi i\omega a^{-1}x}.
\end{align*}
Thus, $G_{\widehat{\tau}}$ has the underlying manifold $(0,\infty)\times \mathbb{R}$, with the group law given by
\begin{equation}
(a,\omega)\ltimes_{\widehat{\tau}}(a',\omega')=(aa',\omega\omega'_a),
\end{equation}
Due to Theorem \ref{T} the left Haar measure $d\mu_{G_{\widehat{\tau}}}(a,\omega)$ is precisely $dad\omega$.
The $\tau\times\widehat{\tau}$-time frequency group $G_{\tau\times\widehat{\tau}}$ has the underlying manifold $(0,\infty)\times \mathbb{R}\times\widehat{\mathbb{R}}$ and the group law is
\begin{equation}
(a,x,\omega)\ltimes_{\tau\times\widehat{\tau}}(a',x',\omega')=(aa',x+ax',\omega\omega'_a),
\end{equation}
with the left Haar measure $d\mu_{G_{\tau\times\widehat{\tau}}}(a,x,\omega)=a^{-1}dadxd\omega$.
If $u\in L^2(\mathbb{R})$ is a window function and also $f\in L^2(G_\tau)$. According to (\ref{CGT1}) we have
\begin{align*}
\mathcal{V}_uf(a,x,\omega)&=\delta(a)^{1/2}V_uf_a(x,\omega)
\\&=a^{-1/2}\langle f_a,\varrho(x,\omega)u\rangle_{L^2(\mathbb{R})}
\\&=a^{-1/2}\int_{-\infty}^{\infty}f(a,y)\overline{[\varrho(x,\omega)u](y)}dy
\\&=a^{-1/2}\int_{-\infty}^{\infty}f(a,y)\overline{u}(y-x)\overline{\omega(y)}dy
=a^{-1/2}\int_{-\infty}^{\infty}f(a,y)\overline{u}(y-x)e^{-2\pi i\omega y}dy.
\end{align*}
Using Theorem \ref{P1}, if $\|u\|_{L^2(\mathbb{R})}=1$ we get
\begin{equation}
\int_{0}^\infty\int_{-\infty}^{\infty}\int_{-\infty}^{\infty}\frac{|\mathcal{V}_uf(a,x,\omega)|^2}{a}dadxd\omega
=\int_0^\infty\int_{-\infty}^{\infty}\frac{|f(a,x)|^2}{a^2}dadx.
\end{equation}
Due to the reconstruction formula (\ref{I10}) if for a.e. $a\in (0,\infty)$ we have $\widehat{f_a}\in L^1(\mathbb{R})$,
then for a.e. $x\in \mathbb{R}$ we get
\begin{align*}
f(a,x)&=\delta(a)^{-1/2}\|u\|_{L^2(K)}^{-2}\int_{-\infty}^{\infty}\int_{-\infty}^{\infty}\mathcal{V}_uf(a,y,\omega)[\varrho(y,\omega)u](x)dyd\omega
\\&=a^{1/2}\|u\|_{L^2(K)}^{-2}\int_{-\infty}^{\infty}\int_{-\infty}^{\infty}\mathcal{V}_uf(a,y,\omega)u(x-y)e^{2\pi i\omega x}dyd\omega.
\end{align*}
As well as according to (\ref{GCGT1})  we have
\begin{align*}
\mathcal{V}_u^\dagger f(a,x,\omega)&=\delta(a)^{1/2}V_uf_a(x^a,\omega_a)
\\&=a^{-1/2}\langle f_a,\varrho(x^a,\omega_a)u\rangle_{L^2(\mathbb{R})}
\\&=a^{-1/2}\int_{-\infty}^{\infty}f(a,y)\overline{[\varrho(x^a,\omega_a)u](y)}dy
\\&=a^{-1/2}\int_{-\infty}^{\infty}f(a,y)\overline{u}(y-ax)\overline{\omega_a(y)}dy
=a^{-1/2}\int_{-\infty}^{\infty}f(a,y)\overline{u}(y-ax)e^{-2\pi i\omega a^{-1}y}dy.
\end{align*}
Using Theorem \ref{P2}, if $\|u\|_{L^2(\mathbb{R})}=1$ we get
\begin{equation}
\int_{0}^\infty\int_{-\infty}^{\infty}\int_{-\infty}^{\infty}\frac{|\mathcal{V}_u^\dagger f(a,x,\omega)|^2}{a}dadxd\omega
=\int_0^\infty\int_{-\infty}^{\infty}\frac{|f(a,x)|^2}{a^2}dadx.
\end{equation}
Due to the reconstruction formula (\ref{I20}) if for a.e. $a\in (0,\infty)$ we have $\widehat{f_a}\in L^1(\mathbb{R})$,
then for $x\in \mathbb{R}$ we achieve
\begin{align*}
f(a,x)&=\int_{-\infty}^{\infty}\int_{-\infty}^{\infty}\mathcal{V}_u^\dagger f(a,y,\omega)[\varrho(y^a,\omega_a)u](x)dyd\omega
\\&=\int_{-\infty}^{\infty}\int_{-\infty}^{\infty}\mathcal{V}_u^\dagger f(a,y,\omega)u(x-ay)e^{2\pi i\omega a^{-1}x}dyd\omega.
\end{align*}
\begin{example}
Let $N>0$ and also $u_N=\chi_{[-N,N]}$ be a window function with compact support and $\|u_N\|_{L^2(\mathbb{R})}=2N$. Then, for all $f\in L^2(G_\tau)$ and $(a,x,\omega)\in G_{\tau\times\widehat{\tau}}$ we have
\begin{align*}
\mathcal{V}_{u_N}f(a,x,\omega)&=a^{-1/2}\int_{-\infty}^{\infty}f(a,y)u_N(y-x)\overline{\omega(y)}dy
\\&=a^{-1/2}\overline{\omega(x)}\int_{-\infty}^{\infty}f(a,y+x)u_N(y)\overline{\omega(y)}dy
\\&=a^{-1/2}\overline{\omega(x)}\int_{-N}^{N}f(a,y+x)\overline{\omega(y)}dy
=a^{-1/2}e^{-2\pi i\omega x}\int_{-N}^{N}f(a,y+x)e^{-2\pi i\omega y}dy.
\end{align*}
If we set $x=0$, then we get
$$\mathcal{V}_{u_N}f(a,0,\omega)=a^{-1/2}\int_{-N}^Nf(a,y)e^{-2\pi i\omega y}dy.$$
Similarly, for the generalized $\tau\times\widehat{\tau}$-continuous Gabor transform we have
\begin{align*}
\mathcal{V}_{u_N}^\dagger f(a,x,\omega)&=a^{-1/2}\int_{-\infty}^{\infty}f(a,y)u_N(y-ax)e^{-2\pi i\omega a^{-1}y}dy
\\&=a^{-1/2}e^{-2\pi i\omega x}\int_{-\infty}^{\infty}f(a,y+ax)u_N(y)e^{-2\pi i\omega a^{-1}y}dy
=a^{-1/2}e^{-2\pi i\omega x}\int_{-N}^{N}f(a,y+ax)e^{-2\pi i\omega a^{-1}y}dy.
\end{align*}
and also if we set $x=0$, then we get
$$\mathcal{V}_{u_N}^\dagger f(a,0,\omega)=a^{-1/2}\int_{-N}^Nf(a,y)e^{-2\pi i\omega a^{-1}y}dy.$$
\end{example}
\begin{example}
Let $u(x)=e^{-\pi x^2}$ be the one-dimensional Gaussian window function with $\widehat{u}=u$ and $\|u\|_{L^2(\mathbb{R})}=2^{-1/4}$. Then, for all $f\in L^2(G_\tau)$ and $(a,x,\omega)\in G_{\tau\times\widehat{\tau}}$ we have
\begin{align*}
\mathcal{V}_uf(x,a,\omega)&=a^{-1/2}\int_{-\infty}^{\infty}f(a,y)u(y-x)e^{-2\pi i\omega y}dy
\\&=a^{-1/2}\int_{-\infty}^{\infty}f(a,y)e^{-\pi(y-x)^2}e^{-2\pi i\omega y}dy.
\end{align*}
If $f$ for a.e. $a\in(0,\infty)$ satisfies $\widehat{f_a}\in L^1(\mathbb{R})$, then we can reconstruct $f$ via
\begin{align*}
f(a,x)&=\int_{-\infty}^{\infty}\int_{-\infty}^{\infty}\mathcal{V}_{u}f(a,y,\omega)u(x-y)e^{2\pi i\omega y}dyd\omega
\\&=\int_{-\infty}^{\infty}\int_{-\infty}^{\infty}\mathcal{V}_{u}f(a,y,\omega)e^{-\pi(x-y)^2}e^{2\pi i\omega x}dyd\omega.
\end{align*}
As well as, we can compute the generalized $\tau\times\widehat{\tau}$-continuous Gabor transform by
\begin{align*}
\mathcal{V}_u^\dagger f(x,a,\omega)
&=a^{-1/2}\int_{-\infty}^{\infty}f(a,y)\overline{u}(y-ax)e^{-2\pi i\omega a^{-1}y}dy
\\&=a^{-1/2}\int_{-\infty}^{\infty}f(a,y)e^{-\pi(y-ax)^2}e^{-2\pi i\omega a^{-1}y}dy.
\end{align*}
If $f$ for a.e. $a\in(0,\infty)$ satisfies $\widehat{f_a}\in L^1(\mathbb{R})$, then we can reconstruct $f$ via
\begin{align*}
f(a,x)&=\int_{-\infty}^{\infty}\int_{-\infty}^{\infty}\mathcal{V}_u^\dagger f(a,y,\omega)u(x-ay)e^{2\pi i\omega a^{-1}x}dyd\omega
\\&=\int_{-\infty}^{\infty}\int_{-\infty}^{\infty}\mathcal{V}_{u}^\dagger f(a,y,\omega)e^{-\pi(x-ay)^2}e^{2\pi i\omega a^{-1}x}dyd\omega.
\end{align*}
\end{example}
The $\tau\otimes\widehat{\tau}$-time frequency group $G_{\tau\otimes\widehat{\tau}}$ has the underlying manifold $(0,\infty)\times(0,\infty)\times \mathbb{R}\times\widehat{\mathbb{R}}$ and the group law is
\begin{equation}
(a,b,x,\omega)\ltimes_{\tau\otimes\widehat{\tau}}(a',b',x',\omega')=(aa',bb',x+ax',\omega\omega'_b),
\end{equation}
with the left Haar measure $d\mu_{G_{\tau\otimes\widehat{\tau}}}(a,b,x,\omega)=a^{-2}dadbdxd\omega$. If $g\in L^2(G_\tau)$ is a window function and also $f\in L^2(G_\tau)$. According to (\ref{CGT2}) we have
\begin{align*}
\mathcal{G}_gf(a,x,b,\omega)&=\delta(b)V_{g_a}f_b(x,\omega)
\\&=b^{-1}\langle f_b,\varrho(x,\omega)g_a\rangle_{L^2(\mathbb{R})}
\\&=b^{-1}\int_{-\infty}^{\infty}f(b,y)\overline{[\varrho(x,\omega)g_a](y)}dy
=b^{-1}\int_{-\infty}^{\infty}f(b,y)\overline{g}(a,y-x)e^{-2\pi i\omega y}dy.
\end{align*}
Using Theorem \ref{P3}, if $\|g\|_{L^2(G_\tau)}=1$ we get
\begin{equation}
\int_0^{\infty}\int_{0}^\infty\int_{-\infty}^{\infty}\int_{-\infty}^{\infty}\frac{|\mathcal{G}_gf(a,x,b,\omega)|^2}{a^2}dadbdxd\omega
=\int_0^\infty\int_{-\infty}^{\infty}\frac{|f(a,x)|^2}{a^2}dadx.
\end{equation}
Using the reconstruction formula (\ref{I30}), if for a.e. $a\in (0,\infty)$ we have $\widehat{f_a},\widehat{g_a}\in L^1(\mathbb{R})$,
then for $x\in \mathbb{R}$ we can write
\begin{align*}
f(b,x)&=b\|g_a\|_{L^2(\mathbb{R})}^{-2}\int_{-\infty}^{\infty}\int_{-\infty}^{\infty}\mathcal{G}_gf(a,y,b,\omega)[\varrho(y,\omega)g_a](x)dyd\omega
\\&=b\|g_a\|_{L^2(\mathbb{R})}^{-2}\int_{-\infty}^{\infty}\int_{-\infty}^{\infty}\mathcal{G}_gf(a,y,b,\omega)g(a,x-y)e^{-2\pi i \omega x}dyd\omega,
\end{align*}
and also in particular we get
\begin{align*}
f(a,x)&=a\|g_a\|_{L^2(\mathbb{R})}^{-2}\int_{-\infty}^{\infty}\int_{-\infty}^{\infty}\mathcal{G}_gf(a,y,a,\omega)g(a,x-y)e^{-2\pi i \omega x}dyd\omega.
\end{align*}
As well as according to (\ref{GCGT2}) we have
\begin{align*}
\mathcal{G}_g^\dagger f(a,x,b,\omega)&=\delta(a)^{-1/2}\delta(b)^{3/2}V_{g_a}f_b(x^a,\omega_b)
\\&=a^{1/2}b^{-3/2}\langle f_b,\varrho(x^a,\omega_b)g_a\rangle_{L^2(\mathbb{R})}
\\&=a^{1/2}b^{-3/2}\int_{-\infty}^{\infty}f(b,y)\overline{[\varrho(x^a,\omega_b)g_a](y)}dy
\\&=a^{1/2}b^{-3/2}\int_{-\infty}^{\infty}f(b,y)\overline{g}(a,y-ax)\overline{\omega_b(y)}dy
=a^{1/2}b^{-3/2}\int_{-\infty}^{\infty}f(b,y)\overline{g}(a,y-ax)e^{-2\pi i\omega b^{-1}y}dy.
\end{align*}
Using Theorem \ref{P4}, if $\|g\|_{L^2(G_\tau)}=1$ we get
\begin{equation}
\int_0^{\infty}\int_{0}^\infty\int_{-\infty}^{\infty}\int_{-\infty}^{\infty}\frac{|\mathcal{G}_g^\dagger f(a,x,b,\omega)|^2}{a^2}dadbdxd\omega
=\int_0^\infty\int_{-\infty}^{\infty}\frac{|f(a,x)|^2}{a^2}dadx.
\end{equation}
Due to the reconstruction formula (\ref{I40}) if for a.e. $a\in (0,\infty)$ we have $\widehat{f_a}\in L^1(\mathbb{R})$,
then for all $x\in \mathbb{R}$ and a.e. $a,b\in (0,\infty)$ we get
\begin{align*}
f(b,x)&=\delta(a)^{-1/2}\delta(b)^{-1/2}\|g_a\|_{L^2(\mathbb{R})}^{-2}\int_{-\infty}^{\infty}\int_{-\infty}^{\infty}\mathcal{G}_g^\dagger f(a,g,b,\omega)[\varrho(g^a,\omega_b)g_a](x)dgd\omega
\\&=a^{1/2}b^{1/2}\|g_a\|_{L^2(\mathbb{R})}^{-2}\int_{-\infty}^{\infty}\int_{-\infty}^{\infty}\mathcal{G}_g^\dagger f(a,g,b,\omega)g(a,x-ag)e^{2\pi i b^{-1}\omega}dgd\omega,
\end{align*}
and also in particular we have
\begin{align*}
f(a,x)&=a\|g_a\|_{L^2(\mathbb{R})}^{-2}\int_{-\infty}^{\infty}\int_{-\infty}^{\infty}\mathcal{G}_g^\dagger f(a,g,a,\omega)[\varrho(g^a,\omega_a)g_a](x)dgd\omega
\\&=a\|g_a\|_{L^2(\mathbb{R})}^{-2}\int_{-\infty}^{\infty}\int_{-\infty}^{\infty}\mathcal{G}_g^\dagger f(a,g,a,\omega)g(a,x-ag)e^{2\pi i a^{-1}\omega}dgd\omega.
\end{align*}
\begin{example}
Let $N>0$ and also $g_N(a,x)=\chi_{[1/N,N]\times[-N,N]}(a,x)$ be a window function which has a compact support. Then, for all $f\in L^2(G_\tau)$ and $(a,x,b,\omega)\in G_{\tau\otimes\widehat{\tau}}$ we have
\begin{align*}
\mathcal{G}_{g_N}f(a,x,b,\omega)&
=b^{-1}\int_{-\infty}^{\infty}f(b,y)g_N(a,y-x)e^{-2\pi i\omega y}dy
\\&=b^{-1}\chi_{[1/N,N]}(a)e^{-2\pi i\omega x}\int_{-N}^{N}f(b,y+x)e^{-2\pi i\omega y}dy.
\end{align*}
If we set $x=0$ and $a=1$ we achieve
\begin{equation}
\mathcal{G}_{g_N}f(1,0,b,\omega)=b^{-1}\int_{-N}^{N}f(b,y)e^{-2\pi i\omega y}dy.
\end{equation}
Also, for $(a,x,b,\omega)$ we have
\begin{align*}
\mathcal{G}_{g_N}^\dagger f(a,x,b,\omega)&=a^{1/2}b^{-3/2}\int_{-\infty}^{\infty}f(b,y)g_N(a,y-ax)e^{-2\pi i\omega b^{-1}y}dy
\\&=a^{1/2}b^{-3/2}\chi_{[1/N,N]}(a)e^{-2\pi i\omega b^{-1}ax}\int_{-N}^{N}f(b,y+ax)e^{-2\pi i\omega b^{-1}y}dy.
\end{align*}
If we set $x=0$ and $a=1$ then
\begin{equation}
\mathcal{G}_{g_N}^\dagger f(1,0,b,\omega)=b^{-3/2}\int_{-N}^{N}f(b,y)e^{-2\pi i\omega b^{-1}y}dy.
\end{equation}
\end{example}
\begin{example}
Let $g(a,x)=ae^{-\pi(a^2+x^2)}$ be the Gaussian type window function in $L^2(G_\tau)$ with $\|g\|_{L^2(G_\tau)}=2^{-1}$. For a.e. $a\in(0,\infty)$ we have $\widehat{g_a}=g_a$ and $\|g_a\|_{L^1(\mathbb{R})}=ae^{-\pi a^2}$ also $\|g_a\|_{L^2(\mathbb{R})}=2^{-1/4}ae^{-\pi a^2}$. It is also separable i.e $g(a,x)=au(a)u(x)$. Then, for all $f\in L^2(G_\tau)$ and also $(a,x,b,\omega)\in G_{\tau\otimes\widehat{\tau}}$ we have
\begin{align*}
\mathcal{G}_gf(a,x,b,\omega)&=b^{-1}\int_{-\infty}^{\infty}f(b,y)g(a,y-x)e^{-2\pi i\omega y}dy
\\&=b^{-1}ae^{-\pi a^2}\int_{-\infty}^{\infty}f(b,y)e^{-\pi(y-x)^2}e^{-2\pi i\omega y}dy.
\end{align*}
Using the reconstruction formula if $\widehat{f_a}\in L^1(\mathbb{R})$ for a.e. $a\in (0,\infty)$ we have
\begin{align*}
f(a,x)&=a\|g_a\|_{L^2(\mathbb{R})}^{-2}\int_{-\infty}^{\infty}\int_{-\infty}^{\infty}\mathcal{G}_{g}f(a,y,a,\omega)g(a,x-y)e^{2\pi i\omega x}dyd\omega
\\&=2^{1/2}a^{-1}e^{\pi a^2}\int_{-\infty}^{\infty}\int_{-\infty}^{\infty}\mathcal{G}_{g}f(a,y,a,\omega)u(x-y)e^{2\pi i\omega x}dyd\omega
\\&=2^{1/2}a^{-1}e^{\pi a^2}\int_{-\infty}^{\infty}\int_{-\infty}^{\infty}\mathcal{G}_{g}f(a,y,a,\omega)e^{-\pi(x-y)^2}e^{2\pi i\omega x}dyd\omega.
\end{align*}
As well as, for all $(a,x,b,\omega)\in G_{\tau\otimes\widehat{\tau}}$ we have
\begin{align*}
\mathcal{G}_g^\dagger f(a,x,b,\omega)&=a^{1/2}b^{-3/2}\int_{-\infty}^{\infty}f(b,y)g(a,y-ax)e^{-2\pi i\omega b^{-1}y}dy
\\&=e^{-\pi a^2}a^{3/2}b^{-3/2}\int_{-\infty}^{\infty}f(b,y)e^{-\pi(y-ax)^2}e^{-2\pi i\omega b^{-1}y}dy.
\end{align*}
Due to the reconstruction formula if $\widehat{f_a}\in L^1(\mathbb{R})$ for a.e. $a\in (0,\infty)$ we have
\begin{align*}
f(a,x)&=a\|g_a\|_{L^2(\mathbb{R})}^{-2}\int_{-\infty}^{\infty}\int_{-\infty}^{\infty}\mathcal{G}_g^\dagger f(a,y,a,\omega)g(a,x-ay)e^{-2\pi i a^{-1}\omega}dyd\omega
\\&=2^{1/2}a^{-1}e^{\pi a^2}\int_{-\infty}^{\infty}\int_{-\infty}^{\infty}\mathcal{G}_g^\dagger f(a,y,a,\omega)u(x-ay)e^{-2\pi i a^{-1}\omega}dyd\omega
\\&=2^{1/2}a^{-1}e^{\pi a^2}\int_{-\infty}^{\infty}\int_{-\infty}^{\infty}\mathcal{G}_g^\dagger f(a,y,a,\omega)e^{-\pi(x-ay)^2}e^{-2\pi i a^{-1}\omega}dyd\omega.
\end{align*}
\end{example}
In the sequel we compute $\tau\times\widehat{\tau}$-time frequency group $G_{\tau\times\widehat{\tau}}$ and $\tau\otimes\widehat{\tau}$-time frequency group $G_{\tau\otimes\widehat{\tau}}$ associate to other semi-direct products.
\subsection{\bf The Weyl-Heisenberg group}
Let $K$ be an LCA group with the Haar measure $dk$ and $\widehat{K}$ be the dual group of $K$ with the Haar measure $d\omega$
also $\mathbb{T}$ be the circle group and let the continuous homomorphism $\tau:K\to Aut(\widehat{K}\times\mathbb{T})$ via $s\mapsto \tau_s$ be given by $\tau_s(\omega,z)=(\omega,z.\omega(s))$. The semi-direct product $G_\tau=K\ltimes_\tau(\widehat{K}\times\mathbb{T})$ is called the Weyl-Heisenberg group associated with $K$. The group operation for all $(k,\omega,z),(k',\omega',z')\in K\ltimes_\tau(\widehat{K}\times\mathbb{T})$ is
\begin{equation}
(k,\omega,z)\ltimes_\tau(k',\omega',z')=(k+k',\omega\omega',zz'\omega'(k)).
\end{equation}
If $dz$ is the Haar measure of the circle group, then $dkd\omega dz$ is a Haar measure for the Weyl-Heisenberg group and also the continuous homomorphism $\delta:K\to (0,\infty)$ given in (\ref{t}) is the constant function $1$. Thus, using Theorem 4.5, Proposition 4.6 of \cite{FollH} and also Proposition \ref{T} we can obtain the continuous homomorphism $\widehat{\tau}:K\to Aut(K\times\mathbb{Z})$
via $s\mapsto\widehat{\tau}_s$, where $\widehat{\tau}_s$ is given by
$\widehat{\tau}_s(k,n)=(k,n)\circ \tau_{s^{-1}}$ for all $(k,n)\in K\times\mathbb{Z}$ and $s\in K$.
Due to Theorem 4.5 of \cite{FollH}, for each $(k,n)\in K\times \mathbb{Z}$ and also for all $(\omega,z)\in \widehat{K}\times\mathbb{T}$ we have
\begin{align*}
\langle(\omega,z),(k,n)_s\rangle&=\langle(\omega,z),\widehat{\tau}_s(k,n)\rangle
\\&=\langle\tau_{s^{-1}}(\omega,z),(k,n)\rangle
\\&=\langle(\omega,z\overline{\omega(s)}),(k,n)\rangle
\\&=\langle\omega,k\rangle\langle z\overline{\omega(s)},n\rangle
\\&=\omega(k)z^n\overline{\omega(s)}^n
\\&=\omega(k-ns)z^n
=\langle\omega,k-ns\rangle\langle z,n\rangle=\langle(\omega,z),(k-ns,n)\rangle.
\end{align*}
Thus, $(k,n)_s=(k-ns,n)$ for all $k,s\in K$ and $n\in\mathbb{Z}$. Therefore, $G_{\widehat{\tau}}$ has the underlying set $K\times K\times\mathbb{Z}$ with the following group operation;
\begin{align*}
(s,k,n)\ltimes_{\widehat{\tau}}(s',k',n')&=\left(s+s',(k,n)\widehat{\tau}_s(k',n')\right)
\\&=\left(s+s',(k,n)(k'-n's,n')\right)=(s+s',k+k'-n's,n+n').
\end{align*}
The $G_{\tau\times\widehat{\tau}}$-time frequency group has the underlying set $K\times \widehat{K}\times\mathbb{T}\times K\times\mathbb{Z}$ with the group law
\begin{equation}
(k,\omega,z,s,n)\ltimes_{\tau\times\widehat{\tau}}(k',\omega',z',s',n')=(k+k',\omega\omega',zz'\omega'(k),s+s'-n'k,n+n').
\end{equation}
and the left Haar measure is $d\mu_{G_{\tau\times\widehat{\tau}}}(k,\omega,z,s,n)=dkd\omega dzdsdn$.

The $G_{\tau\otimes\widehat{\tau}}$-time frequency group has the underlying set $K\times K\times \widehat{K}\times\mathbb{T}\times K\times\mathbb{Z}$ with group operation
\begin{equation}
(r,k,\omega,z,s,n)\ltimes_{\tau\otimes\widehat{\tau}}(r',k',\omega',z',s',n')=(r+r',k+k',\omega\omega',zz'\omega'(r),s+s'-n'k,n+n'),
\end{equation}
and also the left Haar measure is $d\mu_{G_{\tau\otimes\widehat{\tau}}}(r,k,\omega,z,s,n)=drdkd\omega dzdsdn$.
\subsection{Euclidean groups}
Let $E(n)$ be the group of rigid motions of $\mathbb{R}^n$, the group generated by rotations and translations. If we put $H=\SO(n)$ and also $K=\mathbb{R}^n$, then $E(n)$ is the semi direct product of $H$ and $K$ with respect to the continuous homomorphism $\tau:\SO(n)\to Aut(\mathbb{R}^n)$ given by $\sigma\mapsto \tau_\sigma$ via $\tau_\sigma(\bf{x})=\sigma{\mathbf{x}}$ for all $\mathbf{x}\in\mathbb{R}^n$. The group operation for $E(n)$ is
\begin{equation}
(\sigma,\bf{x})\ltimes_{\tau}(\sigma',\bf{x}')=(\sigma\sigma',\bf{x}+\tau_\sigma(\bf{x}'))
=(\sigma\sigma',\bf{x}+\sigma\bf{x}').
\end{equation}
Identifying $\widehat{K}$ with $\mathbb{R}^n$, the continuous homomorphism $\widehat{\tau}:\SO(n)\to Aut(\mathbb{R}^n)$ is given by $\sigma\mapsto\widehat{\tau}_\sigma$ via
\begin{align*}
\langle{\mathbf{x}},\widehat{\tau}_\sigma({\bf{w}})\rangle&=\langle{\bf{x}},{\bf{w}}_\sigma\rangle
\\&=\langle\tau_{\sigma^{-1}}({\bf{x}}),{\bf{w}}\rangle
=\langle\sigma^{-1}{\bf{x}},{\bf{w}}\rangle=e^{-2\pi i (\sigma^{-1}{\bf{x}},{\bf{w}})},
\end{align*}
where $(.,.)$ stands for the standard inner product of $\mathbb{R}^n$. Since $H$ is compact we have $\delta\equiv 1$ and therefore $d\sigma d{\bf x}$ is a left Haar measure for $E(n)$. Thus, $G_{\widehat{\tau}}$ has the underlying manifold $\SO(n)\times\mathbb{R}^n$ with the group operation
\begin{equation}
(\sigma,\bf{w})\ltimes_{\widehat{\tau}}(\sigma',\bf{w}')=(\sigma\sigma',\bf{w}+\bf{w}'_\sigma).
\end{equation}
According to Theorem \ref{T} the left Haar measure $d\mu_{G_{\widehat{\tau}}}(\sigma,{\bf w})$ is precisely $d\sigma d{\bf w}$.
The $\tau\times\widehat{\tau}$-time frequency group $G_{\tau\times\widehat{\tau}}$ has the underlying manifold $\SO(n)\times \mathbb{R}^n\times\mathbb{R}^n$ with the group law
\begin{equation}
(\sigma,{\bf x},{\bf w})\ltimes_{\tau\times\widehat{\tau}}(\sigma',{\bf x}',{\bf w}')=(\sigma\sigma',{\bf x}+\sigma{\bf x}',{\bf w}+{\bf w}'_\sigma).
\end{equation}
Due to Proposition \ref{TT} and also compactness of $H$, $d\mu_{G_{\tau\times\widehat{\tau}}}(\sigma,{\bf x},{\bf w})=d\sigma d{\bf x}d{\bf w}$ is a Haar measure for $G_{\tau\times\widehat{\tau}}$. The $\tau\otimes\widehat{\tau}$-time frequency group $G_{\tau\otimes\widehat{\tau}}$ has the underlying manifold $\SO(n)\times \SO(n)\times\mathbb{R}^n\times\mathbb{R}^n$ equipped with the group law
\begin{align*}
(\sigma,\varrho,{\bf x},{\bf w})\ltimes_{\tau\otimes\widehat{\tau}}(\sigma',\varrho',{\bf x}',{\bf w}')
=(\sigma\sigma',\varrho\varrho',{\bf x}+\sigma{\bf x}',{\bf w}+{\bf w}'_\varrho),
\end{align*}
and also the Haar measure is $d\mu_{G_{\tau\otimes\widehat{\tau}}}(\sigma,\varrho,{\bf x},{\bf w})=d\sigma d\varrho d{\bf x}d{\bf w}$.

\bibliographystyle{amsplain}

\begin{thebibliography}{10}

\bibitem{Dix} Dixmier. J., \textit{$C^{*}$-Algebras}, North-Holland and Publishing company, 1977.

\bibitem{FithStro} Feichtinger. H.G., ${\rm Str\ddot{o}hmer}$. T., \textit{Advances in Gabor Analysis},
Series: Applied and Numerical Harmonic Analysis, ${\rm Birkh\ddot{a}user}$ 2003.

\bibitem{FithZim} Feichtinger. H.G., Zimmermann. G., \textit{A Banach space of test functions for Gabor analysis}, Gabor Analysis and Algorithms: Theory and Applications, p.123-170, Appl. Numer. Harmon. Anal.,
    ${\rm Birkh\ddot{a}user}$ Boston, (1998) Boston, MA.

\bibitem{FollH} Folland. G.B, \textit{A Course in Abstract Harmonic Analysis}, CRC press, 1995.

\bibitem{FollR} Folland. G.B, \textit{Real Analysis, Modern Techniques and their Applications, 2nd ed},
Wiley-Inter-science publication, 1999.

\bibitem{Gab} Gabor. D., \textit{Theory of communication}, J. IEE, {\bf 93}(26), p.429-457, 1964.

\bibitem{Gro1} ${\rm Gr\ddot{o}chenig}$. K., \textit{Foundation of Time-Frequency Analysis}, Appl. Numer. Harmon. Anal., ${\rm Birkh\ddot{a}user}$ Boston, (2001) Boston, MA.

\bibitem{Gro2} ${\rm Gr\ddot{o}chenig}$. K., \textit{Aspects of Gabor Analysis on locally compact abelian groups}, Gabor Analysis and Algorithms: Theory and Applications, p.211-231, Appl. Numer. Harmon. Anal.,
    ${\rm Birkh\ddot{a}user}$ Boston, (1998) Boston, MA.

\bibitem{HeilW} Heil. C., Walnut. F. D., \textit{Continuous and discrete wavelet transforms}, SIAM, SIAM Rev, Vol.31, No 4, p.628-666, 1989.

\bibitem{HR1} Hewitt. E., Ross. R. A., \textit{Abstract Harmonic Analysis}, Vol 1, Springer, No.115, 1963.

\bibitem{HO} Hochschild. G., \textit{The Structure of Lie Groups}, Hpolden-day, San Francisco, 1965.

\bibitem{Lip} Lipsman. R., \textit{Non-Abelian Fourier Analysis},
Bull. Sc. Math., $2^e$ series, {\bf 98}, pp.209-233, 1974.

\bibitem{50} Reiter. H., Stegeman. J. D., \textit{Classical Harmonic Analysis}, 2nd Ed, Oxford University Press, New York, 2000.

\bibitem{Seg1} Segal. I.E., \textit{An extention of Plancherel's formula to seprable unimodular groups}, Ann. of Math, {\bf 52}(1950), pp.272-292.

\bibitem{Ta} Taylor. M.E., \textit{Noncommutative Harmonic Analysis}, Mathematics surveys and Monographs, No 22, American Mathematical Society, Providence, Rhode Island, 1986.

\end{thebibliography}

\end{document}